\documentclass[11pt]{amsart}
\usepackage[colorlinks,linkcolor=blue,anchorcolor=blue,linktocpage=true,urlcolor=blue,citecolor=blue]{hyperref}
\usepackage{amssymb}
\usepackage{mathrsfs}
\usepackage{enumerate}
\usepackage{hyperref}
\usepackage{color}

\newtheorem{theorem}{Theorem}[section]
\newtheorem{definition}{Definition}[section]
\newtheorem{proposition}[theorem]{Proposition}
\newtheorem{lemma}[theorem]{Lemma}
\newtheorem{corollary}[theorem]{Corollary}
\newtheorem{claim}[theorem]{Claim}
\newtheorem{remark}[theorem]{Remark}

\newtheorem{theoremalph}{Theorem}

\def\DD{{\mathbb D}}

 \def\NN{{\mathbb N}} 

 \def\RR{{\mathbb R}}

 \def\ZZ{{\mathbb Z}}

\def \F{\EuScript{F}}

\def \B{\mathcal{B}}

\def \K{\mathcal{K}}

\def \T{\mathcal{T}}

\def\F {{\mathcal F}}
\def\B {{\mathcal B}}
\def\P{\mathcal P}

\def\U{{\mathcal U}}

\begin{document}

\title[Stochastic stability]{Stochastic stability for partially hyperbolic diffeomorphisms with mostly expanding and mostly contracting centers}

%
%

\author[Zeya Mi]{Zeya Mi}

\address{School of Mathematics and Statistics,
Nanjing University of Information Science and Technology, Nanjing 210044, China}
\email{\href{mailto:mizeya@nuist.edu.cn}{mizeya@nuist.edu.cn}, \href{mailto:mizeya@163.com}{mizeya@163.com}}


\date{\today}


\begin{abstract}
We prove the stochastic stability of an open class of partially hyperbolic diffeomorphisms, each of which admits two centers $E^c_1$ and $E^c_2$ such that any Gibbs $u$-state admits only positive (resp. negative) Lyapunov exponents along $E^{c}_1$ (resp. $E^c_2$).
\end{abstract}

\maketitle


\section{Introduction}
\emph{Physical measures} are invariant measures that can reflect chaotic dynamical behaviors with different aspects, they were achieved by Sinai, Ruelle and Bowen \cite{Sin72, Rue76, Bow75, BoR75} for uniformly hyperbolic systems. Given a diffeomorphism $f$ on a compact Riemannian manifold $M$, one says that an $f$-invariant measure $\mu$ is physical if 
\begin{equation*}
\frac{1}{n}\sum_{i=0}^{n-1}\delta_{f^ix}\xrightarrow{weak^*} \mu,\quad n\to +\infty
\end{equation*}
holds on a set of positive Lebesgue measure.

In recent decades, a central topic in smooth dynamical systems is the study on existence and finiteness of physical measures and their stability under perturbations for systems beyond uniform hyperbolicity.
It was conjectured by Palis \cite{pa} that every dynamical system can be approximated by another one having only finitely many physical measures, whose basins cover Lebesgue almost every point of the ambient manifold. Moreover, these physical measures should be stable under perturbations of the system, in stochastic sense.

Denote by ${\rm Diff}^{2}(M)$ the set of $C^2$ diffeomorphisms on $M$.
Given $f\in {\rm Diff}^{2}(M)$, we will study the random dynamical systems generated by i.i.d. random maps around $f$ (see \cite{v97}, \cite{A03},\cite[Appendix D.2]{BDV05}). For that, we consider a metric (parameter) space $\T$ and define the continuous map
\begin{align*}
T:~~\T &  ~~\longrightarrow ~~{\rm Diff}^{2}(M)\\ ~~\omega &~~ \longmapsto ~~f_{\omega}
\end{align*}
with $T(\omega_f)=f$ for some fixed $\omega_f\in \T$.
Let $\{\theta_{\varepsilon}\}_{\varepsilon>0}$ be a family of Borel probabilities of $\T$ such that
\begin{enumerate}[(R1)]
\item\label{R1}
${\rm supp}(\theta_{\varepsilon})$ is a sequence of nested compact subsets satisfying ${\rm supp}(\theta_{\varepsilon})\to\{\omega_{f}\}$ as $\varepsilon\to 0$. 
\item\label{R2} For any $x\in M$, we have $P(x,\cdot)\ll {\rm Leb}$, where $P(x,\cdot)$ is the transition probability defined by $P(x,A)=\theta(\{\omega, f_{\omega}(x)\in A\})$ for any measurable subset $A\subset M$.
\end{enumerate}

We refer to $\{T,\{\theta_{\varepsilon}\}_{\varepsilon>0}\}$ as a \emph{regular random perturbation of} $f$.
Recall that a probability measure $\mu$ is a \emph{stationary} measure with respect to $\theta_{\varepsilon}$ if  
$$
\int \mu\circ f_{\omega}^{-1}d\theta_{\varepsilon}(\omega)=\mu.
$$
We say that $\mu$ is a \emph{zero-noise limit} measure if it is a weak* accumulation point of a sequence of stationary measures $\mu_{\varepsilon}$ with respect to $\theta_{\varepsilon}$ when $\varepsilon\to 0$, it is well known that any zero-noise limit measure is $f$-invariant. 

\begin{definition}
The diffeomorphism $f$ is said to be \emph{stochastically stable} with respect to $\{T,\{\theta_{\varepsilon}\}_{\varepsilon>0}\}$ if any zero noise limit measure is contained in the convex hull of the physical measures of $f$.
\end{definition}

Uniformly hyperbolic systems are known to be stochastically stable \cite{Yuk86,Yuk88,v97,y86}. The knowledge of stochastic stability for systems beyond uniformly hyperbolicity is still very incomplete. However, there are some important results on this subject.
Stochastic stability was obtained for some special kind of non-uniformly hyperbolic systems, like quadratic maps, H\'{e}non maps and Viana maps \cite{KK,by,Blu,BV06}. In the sequence of works \cite{AA03, AAV07,AV13}, stochastic stability was studied for diffeomorphisms with non-uniformly expanding behaviors.
In this paper, we attempt to prove stochastic stability for a a class of partially hyperbolic systems with weak expansion and contraction on centers. 

We say that a diffeomorphism $f$ admits a \emph{dominated splitting} $TM=E_1\oplus_{\succ}\cdots \oplus_{\succ} E_{k}$ if any $E_i$, $1\le i \le k$ is $Df$-invariant and for any $1\le i<j\le k$, we have
$$
\|Df|_{E_i(x)}\|\cdot \|Df^{-1}|_{E_j(f(x))}\|<\frac{1}{2} \quad \textrm{for every}~ x\in M.
$$ 
We use $E^{s}$ (resp. $E^{u}$) to denote that it is the stable sub-bundle which contracts(resp. expands) uniformly under action of $Df$.
A diffeomorphism $f$ is \emph{partially hyperbolic} if there exists a dominated splitting and at least one of the sub-bundles is $E^u$ or $E^s$.

For partially hyperbolic diffeomorphisms admitting $E^{u}$, it was proved in \cite{BDV05} that any physical measure is a Gibbs $u$-state---a measure whose conditional measures along strong unstable manifolds are absolutely continuous with respect to Lebesgue measures. It turns out that for partially hyperbolic systems, Gibbs $u$-states are crucial candidates of physical measures (see e.g. \cite{ps82, bv,bdpp,CY05}).

Denote by $\U(M)$ the set of $C^{2}$ diffeomorphims such that any $f\in \U(M)$ is a partially hyperbolic diffeomorphism with splitting $TM=E^{u}\oplus_{\succ} E^{c}_{1}\oplus_{\succ} E^{c}_2$ such that 
\begin{itemize}
\smallskip
\item
$E^{c}_1$ is mostly expanding: any Gibbs $u$-state has only positive Lyapunov exponents along $E^{cu}$;
\item
\smallskip
$E^{c}_2$ is mostly contracting: any Gibbs $u$-state has only negative Lyapunov exponents along $E^{cs}$.
\end{itemize}

The main result of this paper is as follows:

\begin{theoremalph}\label{TheoA}
Every $f\in \U(M)$ is stochastically stable.
\end{theoremalph}

We emphasize that $\U(M)$ is an open set which contains a large class of partially hyperbolic systems, such as mostly expanding diffeomorphisms and mostly contracting diffeomorphisms. These diffeomorphims has been studied extensively recently. Among which, the existence and finiteness of physical measures for mostly expanding diffeomorphisms and mostly contracting diffeomorphisms were given by Bonatti-Viana \cite{bv} and Andersson-V\'{a}squez \cite{AV15}, respectively. Then, Mi-Cao-Yang \cite{MCY17} gave the same result for diffeomorphisms in $\U(M)$. More recently, another stability called statistical stability was established for these diffeomorphisms \cite{AV17, yjg19, MC20}. 

We would like to remark that partially hyperbolic diffeomorphims introduced in \cite{ABV00} are also sometimes called mostly expanding diffeomorphisms, where the center direction admits the so called non-uniformly expanding behavior. Their stochastic stability was realized earlier by \cite{AAV07} when having persistent non-uniformly expanding behavior. We point out that it has been showed by Andersson-V\'{a}squez \cite[Theorem A]{AV15} that the non-uniformly expanding behavior is not a robust property.

%
%
We will prove Theorem \ref{TheoA} by showing that any zero noise limit measure is a Gibbs $cu$-state. The main step is to establish a sequence of hyperbolic blocks on product space by an argument on the uniformity of central Lyapunov exponents. A pivotal property is that any invariant measure close enough to the zero noise limit measure (lifted to the product space) has uniformly large measure on any fixed hyperbolic block with high level (Theorem \ref{uni2}). This also implies that these invariant measures are random Gibbs $cu$-states (Theorem \ref{gibbse}). By passing limit of the densities of random Gibbs $cu$-states restricted on hyperbolic blocks, one can realize the absolute continuity of zero noise limit measure.

\bigskip
{\bf{Acknowledgements.}} Z. Mi was supported by NSFC 11801278 and he would like to thank Pro. Yongluo Cao for his useful suggestions. 

\section{Preliminary}
Throughout, let $M$ be a compact Riemannian manifold with distance $d(\cdot,\cdot)$. Use ${\rm Leb}$ to denote the Lebesgue measure on $M$. Given a smooth sub-manifold $\gamma$, let ${\rm Leb}_{\gamma}$ represent the Lebesgue measure on $\gamma$ induced by the restriction of the Riemannian structure to $\gamma$.

Given a probability space $(X,\B, \mu)$, $\P$ is called a \emph{measurable partition} of $X$, if there exists a sequence of countable partitions $\{\P_k: k\in \NN\}$ of $X$ such that $\P=\bigvee_{k=0}^{\infty}\P_k$ (mod 0).
Let $\pi:X\to \P$ be the projection which maps each point of $X$ to the element of $\P$ containing it. Denote by $\hat{\mu}$ the quotient measure of $\mu$ w.r.t. $\P$ defined by $\hat{\mu}(A)=\mu(\pi^{-1}A)$ for every measurable subset $A$ of $\P$.
By Rokhlin's disintegration theorem \cite[Appendix C.4]{BDV05}, there exists a unique family of \emph{conditional measures} $\{\mu_{\gamma}: \gamma \in \P\}$ of $\mu$ w.r.t. $\P$ such that 
\begin{itemize}
\smallskip
\item $\mu_{\gamma}(\gamma)=1$ for $\hat{\mu}$-a.e. $\gamma\in \P$, and
\smallskip
\item for any measurable set $B$, $\gamma\mapsto \mu_{\gamma}(B)$ is measurable with 
$$
\mu(B)=\int \mu_{\gamma}(B) d\hat{\mu}.
$$
\end{itemize}

For simplicity, we will say that $\mu$ has \emph{absolutely continuous conditional measures along} $\P$, if $\mu_{\gamma}\ll{\rm Leb}_{\gamma}$ for $\hat{\mu}$-a.e. $\gamma\in \P$.

%

\subsection{SRB measures and Gibbs $cu$-states}
%
Let $f$ be a $C^2$ diffeomorphism on $M$.  Let $\mu$ be an $f$-invariant measure which admits positive Lyapunov exponents almost everywhere. By Pesin theory (see \cite{bp02}), we know that $\mu$-almost every $x$ admits a Pesin unstable manifold $W^{u}(x)$ charactered by
$$
W^u(x)=\Big\{y\in M:\limsup_{n\to +\infty}\frac{1}{n}\log d(f^{-n}(x),f^{-n}(y))<0\Big\}.
$$
A measurable partition $\xi$ is said to be \emph{subordinate to} $W^u$ w.r.t. $\mu$ if for $\mu$-almost every $x$, 
\begin{itemize}
\item $\xi(x)\subset W^u(x)$, where $\xi(x)$ is the element of $\xi$ containing $x$;
\smallskip
\item $\xi(x)$ contains an open neighborhood of $x$ inside $W^u(x)$.
\end{itemize}

One can give the precise definition of SRB measure as follows(see \cite{y02}):

\begin{definition}\label{Def:SRB}
Let $f$ be a $C^2$ diffeomorphism on $M$, an $f$-invariant measure $\mu$ is SRB if
\begin{itemize}
\smallskip
\item[--] it has positive Lyapunov exponents $\mu$-almost everywhere, 
\smallskip
\item[--] $\mu$ has absolutely continuous conditional measures along any measurable partition subordinate to $W^u$.
\end{itemize}
\end{definition}

Let $f$ be a diffeomorphism with a dominated splitting $TM=E^{cu}\oplus_{\succ} E^{cs}$, if $\mu$ is an $f$-invariant measure that has positive Lyapunov exponents along $E^{cu}$, then for
$\mu$-almost every $x$, there exists a Pesin unstable manifold $W^{cu}(x)$ of $x$ tangent to $E^{cu}$ with dimension ${\rm dim}E^{cu}$. Thus one can give the definition of Gibbs $cu$-state using measurable partition subordinate to $W^{E,u}$.

\begin{definition}\label{DCU}
Assume that $f$ is a $C^2$ diffeomorphism with a dominated splitting $TM=E^{cu}\oplus_{\succ} E^{cs}$.
An $f$-invariant measure $\mu$ is called a Gibbs $cu$-state if the Lyapunov exponents of $\mu$ along  $E^{cu}$ are positive and
$\mu$ has absolutely continuous conditional measures along any measurable partition subordinate to $W^{cu}$.
\end{definition}

\begin{lemma}\label{gibbsp1}\cite[Lemma 2.4]{cv}
Let $f$ be a $C^2$ diffeomorphism with dominated splitting $TM=E^{cu}\oplus_{\succ} E^{cs}$. Then almost every ergodic component of any Gibbs $cu$-state is a Gibbs $cu$-state.
\end{lemma}

\begin{lemma}\label{um}\cite[Theorem A]{MCY17}
Every $f\in \U(M)$ admits finitely many ergodic physical measures, they are also Gibbs $cu$-states.
\end{lemma}

\subsection{Random perturbations}

Recall that for given $f\in {\rm Diff}^2(M)$, we define the regular random perturbation $\{T,\{\theta_{\varepsilon}\}_{\varepsilon>0}\}$ around $f$ by considering the continuous map $T:\T\to {\rm Diff}^2(M)$ with $\{\theta_{\varepsilon}\}_{\varepsilon>0}$ having properties (R\ref{R1}),(R\ref{R2}). For any small $\varepsilon>0$, put $\Omega_{\varepsilon}={\rm supp}(\theta_{\varepsilon}^{\ZZ})$, which is a compact subset of $\T^{\ZZ}$.
%

To study the random perturbation, it is useful to consider the random dynamical system on product space $\T^{\ZZ}\times M$ as follows:
\begin{align*}
\F: \T^{\mathbb{Z}}\times M & ~~\longrightarrow ~~\T^{\mathbb{Z}}\times M \\
(\underline{\omega}~,~x) &~~\longmapsto ~~(\sigma(\underline{\omega}),f_{\omega_0}(z)).
\end{align*}
where $\underline{\omega}=(\cdots, \omega_{-1},\omega_{0},\omega_{1}\cdots)\in \T^{\mathbb{Z}}$ and $\sigma$ is the left shift operator.
Let us define the projection map
\begin{align*}
P: \T^{\mathbb{Z}}\times M & ~~\longrightarrow ~~\T^{\mathbb{N}}\times M \\
(\underline{\omega}~,~x) &~~\longmapsto ~~(\underline{\omega}^{+}~, ~x)
\end{align*}
where $\underline{\omega}^{+}=(\omega_0,\omega_1,\cdots)\in \T^{\mathbb{N}}$ for $\underline{\omega}=(\cdots,\omega_{-1},\omega_0,\omega_1,\cdots)\in \T^{\mathbb{Z}}$.

Let $d_{\T}$ be a distance in $\T$, which can generate the distance $d_{\T^{\ZZ}}$ on $\T^{\ZZ}$ defined by
$$
d_{\T^{\ZZ}}(\underline{\omega},\underline{\omega}')=\sum_{n=-\infty}^{+\infty}\frac{d_{\T}(\omega_n,\omega'_n)}{2^{|n|}},
$$
for $\underline{\omega}=(\cdots, \omega_{-1},\omega_{0},\omega_{1}\cdots),\underline{\omega}'=(\cdots, \omega_{-1}',\omega_{0}',\omega_{1}'\cdots)\in \T^{\ZZ}$.
Given $r>0$, denote by
$$
U(\underline{\omega},x;r)=\left\{(\underline{\omega'},x')\in \T^{\ZZ}\times M: d_{\T^{\ZZ}}(\underline{\omega}, \underline{\omega'})\le r; d(x,x')\le r\right\}
$$ 
the compact ball of radius $r$ around $(\underline{\omega},x)$.

For each $\underline{\omega}=(\cdots, \omega_{-1},\omega_{0},\omega_{1}\cdots)\in \T^{\mathbb{Z}}$,
we use the presentation
$$
f_{\underline{\omega}}^n=\left\{
\begin{array}{cl}
  f_{\omega_{n-1}}\circ\cdots \circ f_{\omega_{0}} & ~~n>0~; \\
  id &~~ n=0~;\\
  f_{\omega_n}^{-1}\circ\cdots \circ f_{\omega_{-1}}^{-1} &~~n<0~.
\end{array}\right.
$$

\begin{lemma}\cite[Chapter 1.2, Proposition 1.2]{LiQ95}\label{G}
Let $\mu_{\varepsilon}$ be a stationary measure with respect to $\theta_{\varepsilon}$. Then there is a unique $\F$-invariant measure $\mu_{\varepsilon}^{\ast}$ on $\T^\ZZ\times M$ such that $\mu_{\varepsilon}^*\circ P^{-1}=\theta_{\varepsilon}^\NN\times \mu_{\varepsilon}$.
\end{lemma}

As a consequence of Lemma \ref{G}, we have the following result, see \cite[Lemma 3.4]{MDCDS} for a proof.
\begin{corollary}
Assume that $\mu_{\varepsilon_n}$ is a stationary measure for $\theta_{\varepsilon_n}$ for any $n\in \NN$. If $\mu_n\to\mu$ and $\theta_{\varepsilon_n}\to\theta$ as $n\to +\infty$, then $\mu_{\varepsilon_n}^*\to \mu^*$ as $n\to +\infty$. Furthermore, if we assume $\varepsilon_n\to 0$ as $n\to +\infty$, then we have that 
$$
\lim_{n\to +\infty}\mu_{\varepsilon_n}^*=\delta_{\omega_f}^{\ZZ}\times \mu.
$$
\end{corollary}




\subsection{Random SRB measures and Gibbs $cu$-states}
Let $f\in {\rm Diff}^2(M)$ and fix a stationary measure $\mu$ with respect to $\theta_{\varepsilon}$ for some $\varepsilon>0$. Let $\mu^*$ be the unique $\F$-invariant measure associated to $\mu$ given by Lemma \ref{G}.


By the result of \cite[Chapter VI: Proposition 1.2]{LiQ95}, for $\mu^*$-almost every $(\underline{\omega},x)$, 
there exists the \emph{random Lyapunov exponent} of $(\underline{\omega},x)$ at $v\in T_xM\setminus \{0\}$ defined by
$$
\lim_{n\rightarrow \pm \infty}\frac{1}{n}\log \|Df_{\underline{\omega}}^{n}(x)v\|.
$$
Analogous to the deterministic case, we know that when $(\underline{\omega},x)$ admits positive random Lyapunov exponents, there exists the random unstable manifold $W^{u}(\underline{\omega},x)$, which is a $C^1$ sub-manifold tangent to sub-tangent space at $(\underline{\omega},x)$ spanned by vectors having positive random Lyapunov exponents.

\medskip
A measurable partition $\xi$ is \emph{subordinate} to $W^{u}$ w.r.t. $\mu^*$ if for $\mu^*$-almost every $(\underline{\omega},x)$, 
\begin{itemize}
\item
\smallskip
$\xi(\underline{\omega},x)\subset \{\underline{\omega}\}\times M$, and
\item
\smallskip
$\xi_{\underline{\omega}}(x)=\{y:(\underline{\omega},y)\in \xi(\underline{\omega},x)\}\subset W^{u}(\underline{\omega},x)$,
\smallskip
\item
$\xi_{\underline{\omega}}(x)$ contains an open neighborhood of $x$ inside $W^{u}(\underline{\omega},x)$.
\end{itemize}

Let us recall the definition of random SRB measure as follows.


\begin{definition}
We say that $\mu^*$ is a random SRB measure if 
\begin{itemize}
\item
there exist positive random Lyapunov exponents for $\mu^{\ast}$-almost every $(\underline{\omega},x)$,
\item
for any measurable partition $\xi$ subordinate to $W^{u}$, $\mu^{\ast}$ has absolutely continuous conditional measures along $\xi$.
\end{itemize}
\end{definition}

\begin{lemma}\cite[Theorem B]{LY88}\label{rGibbs}
For a regular random perturbation, if $\mu$ is a stationary measure, then $\mu^*$ is a random SRB measure.
\end{lemma}

By robustness of domination property, we have 
\begin{lemma}\label{fdpd}\cite[Proposition 2.4]{M18S}
Let $f$ be a $C^2$ diffeomorphism admitting the dominated splitting $TM=E^{cu}\oplus_{\succ} E^{cu}$. Then there is $\varepsilon_0>0$ such that there exists the randomly dominated splitting 
$$
T_xM=E^{cu}(\underline{\omega},x)\oplus E^{cs}(\underline{\omega},x),\quad \forall (\underline{\omega},x)\in \Omega_{\varepsilon_0}\times M,
$$
which admits following properties:
\begin{itemize}
\smallskip
\item $Df_{\underline{\omega}}E^{cs}(\underline{\omega},x)=E^{cs}(\F(\underline{\omega},x))$ and $Df_{\underline{\omega}}E^{cu}(\underline{\omega},x)=E^{cu}(\F(\underline{\omega},x))$.
\smallskip
\item both $E^{cu}(\underline{\omega},x)$ and $E^{cs}(\underline{\omega},x)$ depend continuously on $(\underline{\omega},x)$.
\end{itemize}
\end{lemma}

For the rest of this section, we assume that $f$ is a $C^2$ diffeomorphism exhibiting the dominated splitting $TM=E^{cu}\oplus_{\succ} E^{cu}$, and take $\varepsilon_0$ as in Lemma \ref{fdpd}.

Given a stationary measure $\mu$ w.r.t. $\theta_{\varepsilon}$ for $\varepsilon\le \varepsilon_0$.
Let $\mu^*$ be the $\F$-invariant measure associated to $\mu$ given by Lemma \ref{G}. It is known that if $\mu^*$ has only positive random Lyapunov exponents along $E^{cu}$, then there exists random Pesin unstable manifold $W^{cu}(\underline{\omega},x)$ tangent to $E^{cu}$ for $\mu^*$-almost every $(\underline{\omega},x)$.
We say that $\xi$ is a measurable partition \emph{subordinate} to $W^{cu}$ w.r.t. $\mu^{\ast}$ if for $\mu^*$-almost every $(\underline{\omega},x)$, $\xi(\underline{\omega},x)\subset \{\underline{\omega}\}\times M$, and $\xi_{\underline{\omega}}(x)=\{y:(\underline{\omega},y)\in \xi(\underline{\omega},x)\}\subset W^{cu}(\underline{\omega},x)$ contains a neighborhood of $x$ open in $W^{cu}(\underline{\omega},x)$.

\begin{definition}\label{cud}
Let $\mu$ be a stationary measure w.r.t. $\theta_{\varepsilon}$ for $\varepsilon\le \varepsilon_0$, and 
$\mu^*$ is the $\F$-invariant measure associated to $\mu$. We say that $\mu^*$ is a random Gibbs $cu$-state, if all the random Lyapunov exponents of $\mu^*$ along $E^{cu}$ are positive, and for any measurable partition $\xi$ subordinate to $W^{cu}$, $\mu^*$ has absolutely continuous conditional measures along $\xi$. 
\end{definition}



\begin{lemma}\label{gibbsp}
Let $\mu$ be an $f$-invariant measure, then it is a Gibbs $cu$-state of $f$ if and only if $\mu^{\ast}=\delta_{\omega_f}^{\ZZ}\times \mu$ is a random Gibbs $cu$-state.
\end{lemma}

\section{A criterion on random Gibbs $cu$-states}
Throughout this section, we assume that $f$ is a $C^2$ diffeomorphism exhibiting the dominated splitting $TM=E^{cu}\oplus_{\succ} E^{cs}$. 
Let us fix $\varepsilon_0$ as in Lemma \ref{fdpd}, thus there exists the random dominated splitting $E^{cu}(\underline{\omega},x)\oplus E^{cs}(\underline{\omega},x)$ over $T_xM$ for $(\underline{\omega},x)\in \Omega_{\varepsilon_0}\times M$.
For simplicity, we say that $\mu$ is a stationary measure of $\{T,\{\theta_{\varepsilon}\}_{0<\varepsilon\le \varepsilon_0}\}$ if it is a stationary measure w.r.t $\theta_{\varepsilon}$ for some $\varepsilon\le \varepsilon_0$.

\medskip
Given $\alpha>0$ and $\ell\in \NN$, we refer to the compact subset 
\begin{align*}
&\B_{\ell}(\alpha)=\Big\{(\underline{\omega},x)\in \Omega_{\varepsilon_0}\times M: \prod_{i=0}^{n-1}\|Df^{\ell}_{\sigma^{i\ell}\underline{\omega}}|_{E^{cs}(\F^{i\ell}(\underline{\omega},x))}\|\le {\rm e}^{-\alpha \ell n},\\
&\quad\quad\quad\quad\quad\quad\quad ~\prod_{i=0}^{n-1}\|Df^{-\ell}_{\sigma^{-i\ell}\underline{\omega}}|_{E^{cu}(\F^{-i\ell}(\underline{\omega},x))}\|\le {\rm e}^{-\alpha \ell n},\quad \forall n\in \NN\Big\}.
\end{align*}
as a \emph{hyperbolic block} of level $\ell$ w.r.t. $\alpha$.

The main goal of this section is to show the following result, whose proof will be given in $\S$\ref{se}.
\begin{theorem}\label{gibbse}
Let $\mu_n$ be a sequence of stationary measures of $\{T,\{\theta_{\varepsilon}\}_{0<\varepsilon\le \varepsilon_0}\}$ that converges to $\mu$. If for every $\delta>0$, there exists $\ell\in \NN$ such that
$$
\limsup_{n\to +\infty}\mu^{\ast}_n(\B_{\ell}(\alpha))>1-\delta,
$$
then $\mu^{\ast}$ is a random Gibbs $cu$-state.
\end{theorem}

\subsection{Foliated charts associated to stack of unstable manifolds}
%
%
%
%
%
%
Let $\DD^{cu}$ be the ${\rm dim}E^{cu}$-dimensional compact unit ball of $\RR^{{\rm dim}E^{cu}}$. Let ${\rm Emb^1}(\DD^{cu},M)$ be the space of $C^1$ embeddings from $\DD^{cu}$ to $M$.

Let us recall the following folklore result which can be deduced from the classic Plaque family theorem \cite[Theorem 5.5]{HPS77}, see \cite[Theorem 3.5]{AV17} also.
\begin{proposition}\label{pu}
Given $\alpha>0$ and $\ell \in \NN$, there exist $C_{\ell,\alpha}>0$, $r_{\ell,\alpha}>0$, and a continuous map $\Theta: \B_{\ell}(\alpha)\to {\rm Emb^1}(\DD^{cu},M)$ such that for each $(\underline{\omega},x)\in \B_{\ell}(\alpha)$, we have
\begin{itemize}
\item $\Theta(\underline{\omega},x)(\DD^{cu}):=W^u_{r_{\ell,\alpha}}(\underline{\omega},x)$ has radius $r_{\ell,\alpha}$ centered at $x$;
\smallskip
\item for each $y\in W^u_{r_{\ell,\alpha}}(\underline{\omega},x)$, it has that
$$
T_yW^u_{r_{\ell,\alpha}}(\underline{\omega},x)=E^{cu}(\underline{\omega},y);
$$
\item if $y\in W^u_{r_{\ell,\alpha}}(\underline{\omega},x)$, then 
$$
d(f_{\underline{\omega}}^{-n}(x), f_{\underline{\omega}}^{-n}(y))\le C_{\ell,\alpha}{\rm e}^{-\alpha\ell n/2}d(x,y),\quad \forall n\in \NN.
$$
\end{itemize}

\end{proposition}

We will say that $W^u_{r_{\ell,\alpha}}(\underline{\omega},x)$ is a local random unstable manifold (of size $r_{\ell,\alpha}$) of $(\underline{\omega},x)$. It can also be seen as a disk in the product space $\Omega_{\varepsilon_0}\times M$, by identifying $\{\underline{\omega}\}\times W^u_{r_{\ell,\alpha}}(\underline{\omega},x)$ with $W^u_{r_{\ell,\alpha}}(\underline{\omega},x)$.

\begin{remark}
Actually, the existence of $W^u_r(\underline{\omega},x)$ is deduced only by domination and the backward contracting property on $E^{cu}$. The domination and forward contracting property of $E^{cs}$ imply the existence of local random stable manifolds tangent to $E^{cs}$, though we do not use this fact.
\end{remark}

We have the Lipschitz continuity of random local unstable manifolds, which is a generalization of its deterministic version (see e.g. \cite[$\S$ 6.1]{HP68}). One can see \cite[Appendix A]{MDCDS} for a direct proof.
\begin{lemma}\label{hee}
Given $\alpha>0$ and $\ell\in \NN$.
There exists a constant $L>0$ such that for any $(\underline{\omega},x)\in \B_{\ell}(\alpha)$,
we have
$$
\rho\left(T_yW^u_{r_{\ell,\alpha}}(\underline{\omega},x),T_zW^u_{r_{\ell,\alpha}}(\underline{\omega},x)\right)\le Ld(y,z)
$$
whenever $y,z\in W^u_{r_{\ell,\alpha}}(\underline{\omega},x)$.
\end{lemma}

%
\medskip
Let us fix $\alpha>0$ now. For every $\ell\in \NN$, let $\B_{\ell}:=\B_{\ell}(\alpha)$ and $r_{\ell}:=r_{\ell,\alpha}$ for simplicity. 
Take $r\ll r_{\ell}$, for every $(\underline{\omega},x)\in \Omega_{\varepsilon_0}\times M$, define
$$
\P_{\ell}(\underline{\omega},x; r):=\left\{W^u_r(\underline{\omega}',x')\cap U(\underline{\omega},x;r): (\underline{\omega'},x')\in U(\underline{\omega},x; r/2)\cap \B_{\ell}\right\},
$$
which is a stack of random local unstable manifolds.
We denote by $S_{\ell}(\underline{\omega},x;r)$ the union of elements of $\P_{\ell}(\underline{\omega},x;r)$. 

\begin{lemma}\label{eui}
Assume that $\mu$ is an $\F$-invariant measure.
If $\B_{\ell_{n}}$ is a sequence of hyperbolic blocks such that 
\begin{enumerate}[1.]
\smallskip
\item
$\mu(\bigcup_{n=1}^{\infty}\B_{\ell_n})=1,$ 
\smallskip
\item\label{t} for every $\ell_n$, for every $(\underline{\omega},x)\in {\rm supp}(\mu|\B_{\ell_n})$, $\mu$ has absolutely continuous conditional measures along $\P_{\ell_n}(\underline{\omega},x;r)$ for some $r\ll r_{\ell_n}$,
\end{enumerate}
then $\mu$ is a random Gibbs $cu$-state.
\end{lemma}

Lemma \ref{eui} maybe seen as a definition of the random Gibbs $cu$-state, which is equivalent to Definition \ref{cud}, one can see a discussion on this equivalence in \cite[Section 7.1]{AL}.
It indicates that one can describe the absolute continuity of Gibbs $cu$-states restricted on stacks of random local unstable manifolds. 

To describe the absolute continuity of conditional measures along random unstable manifolds in local sense, it is convenient to introduce the notion of foliated chart.
 
\begin{definition}\label{Def:foliated-chart}
Let $\K$ be a compact metric space. A \emph{foliated chart} associated to $\K$ is a homeomorphism $\Phi:~\K \times \DD^{k}\rightarrow \B$ such that
\begin{itemize}

\item $\Phi_p=\Phi|\{p\}\times \DD^k$ is a diffeomorphism for each $p\in \K$.

\item $\Phi_p$ maps $\DD^k$ to disjoint (endowed random) local unstable disks with dimension $k$.

\item $p\mapsto \Phi_p$ depends continuously in $C^1$ topology.
\end{itemize}
\end{definition}

The result below demonstrates the existence of foliated charts associated to stack of random local unstable manifolds.

\begin{lemma}\label{foli}
Let $\ell\in\NN$ and $\mu$ an $\F$-invariant measure such that $\mu(\B_{\ell})>0$. Then there exists $\delta_{\ell}\le r_{\ell}$ such that for any $(\underline{\omega},x)\in {\rm supp}(\mu|{\B_{\ell}})$, for every $r\le \delta_{\ell}$, 
we have $\mu(S_{\ell}(\underline{\omega},x;r))>0$ and
there exists a foliated chart $\Phi: X\times \DD^{cu}\to S_{\ell}(\underline{\omega},x;r)$ for some compact subset $X$.
\end{lemma}

\begin{proof}
Let us fix $(\underline{\omega},x)\in {\rm supp}(\mu|{\B_{\ell}})$.
It follows from Proposition \ref{pu} that for any $0<r \ll r_{\ell}$, the map
\begin{align*}
h_{r}: U(\underline{\omega},x; r/2)\cap \B_{\ell} & ~~\longrightarrow ~~\T^{\ZZ}\times M \\
(\underline{\omega'}~,~x') &~~\longmapsto ~~W^u_{r_{\ell}}(\underline{\omega}',x')\cap U(\underline{\omega},x;r)
\end{align*}
is continuous in $C^1$-topology. Using this continuity one can choose $\delta_{\ell}\ll r_{\ell}$, and take a smooth compact disk $\gamma$ containing $x$ with following property: for every $r \le\delta_0$, if we take $\Gamma=\Omega_{\varepsilon_0}\times \gamma$, then for any $(\underline{\omega'},x')\in U(\underline{\omega},x; r/2)\cap \B_{\ell}$, $W^u_{r_{\ell}}(\underline{\omega'},x')$ is transverse to $\Gamma$ at a single point, denoted by $\tau(\underline{\omega'},x')$.
Writing 
$$
X:=\bigcup_{(\underline{\omega'},x')\in U(\underline{\omega},x; r/2)\cap \B_{\ell}}\tau(\underline{\omega'},x').
$$ 

\begin{claim} 
$X$ is a compact subset of $\Omega_{\varepsilon_0}\times M$.
\end{claim}
\begin{proof}
Let $(\underline{\omega_n},x_n)\in X$ be a sequence of points that converges to $(\underline{\omega'},x')\in \Omega_{\varepsilon_0}\times M$ as $n\to +\infty$. By the compactness of $\Gamma$, we have $(\underline{\omega'},x')\in \Gamma$. 
By definition of $\tau$, for every $n\in \NN$ we may choose $(\underline{\omega}_n,\widehat{x}_n)\in U(\underline{\omega},x; r/2)\cap \B_{\ell}$ such that
$
\tau(\underline{\omega}_n,\widehat{x}_n)=(\underline{\omega}_n,x_n).
$ 
Up to considering subsequence, we may assume that $(\underline{\omega}_n,\widehat{x}_n)$ converges to a point $(\underline{\omega'},\widehat{x})\in U(\underline{\omega},x; r/2)\cap \B_{\ell}$. By the continuity of $h_r$, we know 
$(\underline{\omega}_n,x_n)$ converges to a point in $W^u_{r_{\ell}}(\underline{\omega}',\widehat{x})\cap U(\underline{\omega},x;\varepsilon)$. As $(\underline{\omega}_n,x_n)\to(\underline{\omega'},x')$, we have $(\underline{\omega'},x')\in W^u_{r_{\ell}}(\underline{\omega}',\widehat{x})\cap U(\underline{\omega},x;r)$. 
To conclude, we have shown that there exists $(\underline{\omega'},\widehat{x})\in U(\underline{\omega},x; r/2)\cap \B_{\ell}$ such that $(\underline{\omega}_n, x_n)\to \tau(\underline{\omega'},\widehat{x})$. This gives the compactness of $X$. 
\end{proof}

By the continuity of $h_{r}$, one can construct the continuous map $\phi: \Gamma\to {\rm Emb}^1(\DD^{cu},\Omega_{\varepsilon_0}\times M)$ defined by
$$
\phi(\underline{\omega}',y)(\DD^{cu})=W^u_{r_{\ell}}(\underline{\omega}',\widehat{y})\cap U(\underline{\omega},x;r),
$$
where $(\underline{\omega}',\widehat{y})$ is a point of $U(\underline{\omega},x; r/2)\cap \B_{\ell}$ satisfying $\tau(\underline{\omega}',\widehat{y})=(\underline{\omega}', y)$. Define
\begin{align*}
\Phi: \Gamma\times \DD^{cu} & ~~\longrightarrow ~~S_{\ell}(\underline{\omega},x;r) \\
(\underline{\omega'}~,~y,~z) &~~\longmapsto ~~\phi(\underline{\omega}',y)(z)
\end{align*}
According to the definition of $\phi$, one knows that $\Phi$ is a foliated chart.
\end{proof}

\medskip
Given an $\F$-invariant measure $\mu$ such that $\mu(S_{\ell}(\underline{\omega},x))>0$. If $\Phi: X\times \DD^{cu}\to S_{\ell}(\underline{\omega},x)$ is a foliated chart, then we know $\mu$
has absolutely continuous conditional measures on $\mathcal{P}_{\ell}(\underline{\omega},x)$ iff for the pullback $\nu:=\mu \circ \Phi$, there exists a measurable function $\rho:X\times \DD^{cu}\longrightarrow [0,\infty)$ such that
$$
\nu(A)=\int_A \rho(x,y)d{\rm Leb}_{\DD^{cu}}(y)d\hat{\nu}(x)
$$
for every measurable subset $A$ of $X \times \DD^{cu}$, where $\hat{\nu}$ is the quotient measure of $\nu$ defined by $\hat{\nu}(\xi)=\nu(\xi\times \DD^{cu})$ for every measurable $\xi \subset X$.

\medskip
We will use the following argument (see e.g. \cite[Proposition 7.3]{v1}).
\begin{lemma}\label{vp} 
Let $\mu$ be an $\F$-invariant measure such that $\mu(S_{\ell}(\underline{\omega},x;\varepsilon))>0$ for some $\varepsilon>0$, assume that $\Phi: X\times \DD^{cu}\to S_{\ell}(\underline{\omega},x;\varepsilon)$ is a foliated chart. Then for every $C>0$, there exists $C'=C'(C,\Phi)>0$ with following property: 

If for every open subset $\zeta\subset X$ satisfying $\hat{\nu}(\partial{\zeta})=0$ and open subset $\eta \subset \DD^{cu}$, one has
$$
\nu(\zeta\times \eta)\le C\cdot {\rm Leb}_{\DD^{cu}}(\eta)\hat{\nu}(\zeta),
$$
where $\nu:=\mu \circ \Phi$,
then $\mu$ has absolutely continuous conditional measures along $\P_{\ell}(\underline{\omega},x;\varepsilon)$ with density bounded from above by $C'$.
\end{lemma}

\subsection{Proof of Theorem \ref{gibbse}}\label{se}
Let $\mu_n$ be a sequence of stationary measures for which $\mu_n\to \mu$ as $n\to +\infty$. Lemma \ref{G} implies that $\mu_n^*\to \mu^*$ as $n\to +\infty$.
By assumption and compactness of hyperbolic bolcks, for any $m\in \NN$, there exists $\ell_m\in \NN$ such that 
\begin{equation}\label{est1}
\mu^*(\B_{\ell_{m}}(\alpha))\ge \limsup_{n\to+\infty}\mu^*_n(\B_{\ell_m}(\alpha))>1-\frac{1}{m}.
\end{equation}
Thus, we have 
$$
\mu^*\left(\bigcup_{m\ge 1}\B_{\ell_m}(\alpha)\right)=1.
$$
By Lemma \ref{eui}, it suffices to verify that $\B_{\ell_m}(\alpha)$ satisfies Condition \ref{t} of Lemma \ref{eui}.

Fixing integers $m$ and $\ell_m$ as in (\ref{est1}).  Up to taking subsequences, we may assume that   
$$
\mu_n(\B_{\ell_m}(\alpha))>1-\frac{1}{m}.
$$
holds for every $n\in \NN$.
Choose any $(\underline{\omega},x)\in {\rm supp}(\mu|\B_{\ell_n}(\alpha))$, and then take $r$ small enough and satisfies
\begin{equation}\label{se}
\mu\left(\partial(S_{\ell_m}(\underline{\omega},x;r))\right)=0,
\end{equation}
recall that $S_{\ell_m}(\underline{\omega},x;r)$ is the union of stack of random local unstable manifolds with uniform size and backward contracting rate.

It follows from Lemma \ref{rGibbs} that any $\mu^*_n$ is a random SRB measure. According to the definition of hyperbolic blocks, we have that $\mu_n^*$-almost every point of $\B_{\ell_m}(\alpha)$ has only negative Lyapunov exponents along $E^{cs}$ and positive Lyapunov exponents along $E^{cu}$. Denote by $\widehat{\rho}_n$ be the density of the conditional measures of $\mu_n$ along $\P_{\ell_m}(\underline{\omega},x;r)$, put $J^{cu}(\underline{\omega'},y)=|{\rm det}Df_{\underline{\omega'}}|_{E^{cu}(\underline{\omega'},y)}|$. Then, we have \cite[ Chapter VI: Proposition 2.2 and Corollary 8.1]{LiQ95}
$$
\frac{\widehat{\rho}_n(\underline{\omega'},y)}{\widehat{\rho}_n(\underline{\omega'},z)}=\prod_{k=1}^{\infty}\frac{J^{cu}(\F^{-k}(\underline{\omega'},z))}
{J^{cu}(\F^{-k}(\underline{\omega'},y))}
$$
for $\mu_{\{\underline{\omega'}\}\times \gamma}$-almost every $(\underline{\omega'},y)$ and $(\underline{\omega'},z)$ in $\{\underline{\omega'}\}\times \gamma\in \P_{\ell_m}(\underline{\omega},x;r)$. Using Lemma \ref{hee}, one gets that there exists $C_0$ independent of $n$ such that for every $n\in \NN$,
$$
\frac{1}{C_0}\le \widehat{\rho}_n(\underline{\omega'},y)\le C_0
$$
for $\mu_{\{\underline{\omega'}\}\times \gamma}$-almost every $(\underline{\omega'},y)$ in $\{\underline{\omega'}\}\times \gamma\in \P_{\ell_m}(\underline{\omega},x;r).$
By Lemma \ref{foli}, there exists a foliated chart
$
\Phi: X\times \DD^{cu}\to S_{\ell_m}(\underline{\omega},x;r).
$
Consequently, there is a constant $C_1>0$ such that for any $n\in \NN$, for measurable subsets $\zeta\subset X$ and $\eta\subset \DD^{cu}$, 
\begin{equation}\label{fdcd}
\nu_n(\zeta \times \eta)\le C_1 \cdot {\rm Leb}_{\DD^{cu}}(\eta) \hat{\nu}_n(\zeta)\quad \textrm{for every}~n\in \NN,
\end{equation}
where $\nu_n=\mu_n^* \circ \Phi $ for every $n\in \NN$. Define $\nu=\mu^* \circ \Phi $, then the convergence $\mu_n^*\to \mu^*$ yields $\nu_n \to \nu$ directly. By assumption (\ref{se}), we have $\hat{\nu}_n\to \hat{\nu}$ as $n\to \infty$. Take any open subset $\zeta \subset \DD^{cu}$ and open subset $\zeta \subset X$ with $\hat{\nu}(\partial{\zeta})=0$.  Applying (\ref{fdcd}), one gets
\begin{eqnarray*}
\nu(\zeta \times \eta) &\le & \liminf_{n\to \infty}\nu_n(\zeta \times \eta)\\
&\le & C\cdot{\rm Leb}_{\DD^{cu}}(\eta)\liminf_{n\to \infty}  \hat{\nu}_n(\zeta)\\
&=& C\cdot{\rm Leb}_{\DD^{cu}}(\eta)\hat{\nu}(\zeta).
\end{eqnarray*}
By Lemma \ref{vp}, one knows that $\mu$ has absolutely continuous conditional measures along $\P_{\ell_m}(\underline{\omega},x;r)$.

\section{Proof of the Theorem \ref{TheoA}}
Let $f$ be a $C^2$ diffeomorphism exhibiting a dominated splitting $TM=E^{cu}\oplus_{\succ} E^{cs}$. Let us fix the noise level $\varepsilon_0$ given by Lemma \ref{fdpd}.
Given $\ell\in \NN$ and $\alpha>0$, let
$$
K_{\ell}^{u}(\alpha)=\Big\{(\underline{\omega},x)\in \Omega_{\varepsilon_0}\times M: \prod_{i=0}^{n-1}\|Df^{\ell}_{\underline{\omega}}|_{E^{cs}(\underline{\omega},x)}\|\le {\rm e}^{-\alpha \ell}, \forall n\ge 1 \Big\};
$$

$$
K_{\ell}^{s}(\alpha)=\Big\{(\underline{\omega},x)\in \Omega_{\varepsilon_0}\times M: \prod_{i=0}^{n-1}\|Df^{-\ell}_{\underline{\omega}}|_{E^{cu}(\underline{\omega},x)}\|\le {\rm e}^{-\alpha \ell}, \forall n\ge 1 \Big\}.
$$
Then define
$$
K_{\ell}(\alpha)=K_{\ell}^{u}(\alpha)\cap K_{\ell}^{s}(\alpha).
$$

\begin{lemma}\label{ex}
Given $0<\alpha<\beta$, if $\{\mu_n\}$ is a sequence of $\F$-invariant measures that converges to $\mu$ such that
$$
\lim_{n\to+\infty}\frac{1}{n}\log \|Df_{\underline{\omega}}^{-n}|_{E^{cu}(\underline{\omega},x)}\|<-\beta,
\quad
\lim_{n\to+\infty}\frac{1}{n}\log \|Df_{\underline{\omega}}^{n}|_{E^{cs}(\underline{\omega},x)}\|<-\beta
$$
for $\mu$-almost every $(\underline{\omega},x)\in \Omega_{\varepsilon_0}\times M$,
then for any $\eta\in (0,1)$ there exists $\ell(\eta)\in \NN$ such that for any $\ell\ge \ell(\eta)$, we have
$$
\liminf_{n\to +\infty}\mu_n\left(K_{\ell}(\beta)\right)>1-\eta.
$$
\end{lemma}

\begin{proof}
Since 
$$
\lim_{n\to+\infty}\frac{1}{n}\log \|Df_{\underline{\omega}}^{-n}|_{E^{cu}(\underline{\omega},x)}\|<-\beta
$$
holds on a subset with full $\mu$-measure, for any $\eta\in (0,1)$ there exists $\ell(\eta)\in \NN$ such that 
for any $\ell\ge \ell(\eta)$, we have
$$
\mu\left(U_{\ell}\right)>1-\eta/2, 
$$
where
$$
U_{\ell}:=\Big\{(\underline{\omega},x)\in \Omega\times M: \frac{1}{\ell}\log \|Df_{\underline{\omega}}^{-\ell}|_{E^{cu}(\underline{\omega},x)}\|<-\beta\Big\}.
$$
By definition, $U_{\ell}$ is open and contained in $K_{\ell}^{u}(\beta)$, thus the convergence $\mu_n\to \mu$ suggests that 
$$
\liminf_{n\to +\infty}\mu_n(K_{\ell}^{u}(\beta))\ge \liminf_{n\to +\infty}\mu_n(U_{\ell})>1-\eta/2.
$$
By the similar argument as above, we conclude that
$$
\liminf_{n\to +\infty}\mu_n(K_{\ell}^{s}(\beta))>1-\eta/2.
$$
Consequently, we have
$$
\liminf_{n\to +\infty}\mu_n(K_{\ell}(\beta))>1-\eta.
$$
\end{proof}

By using Pliss-like Lemma (see \cite[Lemma 5.8]{CMY18}, \cite[Lemma A]{AV17}) to Lemma \ref{ex}, we can obtain the following result. The proof is similar to \cite[Proposition 5.6]{CMY18}, hence we omit here.

\begin{theorem}\label{uni2}
Given $0<\alpha<\beta$, if $\{\mu_n\}$ is a sequence of $\F$-invariant measures that converges to $\mu$ such that
$$
\lim_{n\to +\infty}\frac{1}{n}\log \|Df_{\underline{\omega}}^n|_{E^{cs}(\underline{\omega},x)}\|<-\beta,\quad
\lim_{n\to +\infty}\frac{1}{n}\log \|Df_{\underline{\omega}}^{-n}|_{E^{cu}(\underline{\omega},x)}\|<-\beta
$$
for $\mu$-almost every $(\underline{\omega},x)\in \Omega_{\varepsilon_0}\times M$,
then we have 
$$
\lim_{\ell\to +\infty}\liminf_{n\to +\infty}\mu_n(\B_{\ell}(\alpha))=1.
$$
\end{theorem}

As a consequence of Theorem \ref{uni2}, we have

\begin{corollary}\label{uni}
Let $f$ be a diffeomorphism with a dominated splitting $TM=E^{cu}\oplus E^{cs}$. Given $\alpha>0$, if $\mu$ is a zero noise limit of a sequence of stationary measures $\mu_n$ satisfying
$$
\lim_{n\to +\infty}\frac{1}{n}\log \|Df^n|_{E^{cs}(x)}\|<-\beta,\quad
\lim_{n\to +\infty}\frac{1}{n}\log \|Df^{-n}|_{E^{cu}(x)}\|<-\beta
$$
holds for $\mu$-almost every $x\in M$,
then we have 
$$
\lim_{\ell\to +\infty}\liminf_{n\to +\infty}\mu_n^*(\B_{\ell}(\alpha))=1.
$$

\end{corollary}

\begin{proof}[Proof of Corollary \ref{uni}]
Assume that $\mu_n\to \mu$ as $n\to +\infty$, then $\mu$ is an $f$-invariant measure on $M$. By Lemma \ref{G}, we obtain that $\mu^*=\delta_{\omega_f}^{\ZZ}\times \mu$. The assumption 
on Lyapunov exponents along $E^{cs}$ and $E^{cu}$ is equivalent to say that 
$$
\lim_{n\to +\infty}\frac{1}{n}\log \|Df_{\underline{\omega}}^n|_{E^{cs}(\underline{\omega},x)}\|<-\beta,\quad
\lim_{n\to +\infty}\frac{1}{n}\log \|Df_{\underline{\omega}}^{-n}|_{E^{cu}(\underline{\omega},x)}\|<-\beta
$$
hold for $\mu^*$-almost every $(\underline{\omega},x)\in \Omega_{\varepsilon_0}\times M$. Therefore, we get the desired result by applying Theorem \ref{uni}.
\end{proof}

Now we are ready to complete the proof of Theorem \ref{TheoA}.
\begin{proof}[Proof of Theorem \ref{TheoA}]
Let $f\in \U(M)$ with partially hyperbolic splitting $TM=E^u\oplus E_1^c\oplus E_2^c$ such that $E_1^c$ is mostly expanding and $E_2^c$ is mostly contracting. Given $\alpha>0$ and $\ell\in \NN$, recall the definition of hyperbolic blocks $\{\B_{\ell}(\alpha)\}_{\ell\ge 1}$, where we assume $E^{cu}:=E^u\oplus E_1^c$ and $E^{cs}:=E_2^c$.

By the result of \cite[Lemma 3.4 \& Lemma 3.5]{MCY17}, one knows that for every Gibbs $u$-state $\mu$ of $f$, there is $\beta>0$ such that 
$$
\lim_{n\to +\infty}\frac{1}{n}\log \|Df^n|_{E^{cs}(x)}\|<-\beta,\quad
\lim_{n\to +\infty}\frac{1}{n}\log \|Df^{-n}|_{E^{cu}(x)}\|<-\beta
$$
for $\mu$-almost every $x\in M$. Now we assume that $\mu$ is a zero noise limit of some sequence of stationary measures $\mu_n$, which implies that $\mu$ is a Gibbs $u$-state(see e.g. \cite[Proposition 5]{CY05}). Let $\mu^*$ and $\mu_n^*$ be the  $\F$-invariant measures associated to $\mu_n$ and $\mu$ given by Lemma \ref{G}.
By Corollary \ref{uni}, for any $\alpha>\beta$ we have 
$$
\lim_{\ell\to +\infty}\liminf_{n\to +\infty}\mu_n^*(\B_{\ell}(\alpha))=1.
$$
Hence, Theorem \ref{gibbse} implies that $\mu^*$ is a Gibbs $cu$-state for the extended dynamical system. Using  Lemma \ref{gibbsp}, we know $\mu$ is a Gibbs $cu$-state. By Lemma \ref{gibbsp1} and Lemma \ref{um}, $\mu$ is a convex combination of finitely many ergodic physical measures of $f$. This completes the proof of Theorem \ref{TheoA}.
\end{proof}

\end{document}